\documentclass[oneside]{amsart}

\usepackage{graphicx}%
\usepackage{multirow}%
\usepackage{amsmath,amssymb,amsfonts}%
\usepackage{amsthm}%

\theoremstyle{plain}
\newtheorem{thm}{Theorem}[section]
\newtheorem{cor}[thm]{Corollary}
\newtheorem{definition}[thm]{Definition}
\newtheorem{prop}[thm]{Proposition}

\begin{document}

\title{Arndt and De Morgan Integer Compositions}

\author{Brian Hopkins}\address{Saint Peter's University, Jersey City NJ, USA}\email{bhopkins@saintpeters.edu}

\author{Aram Tangboonduangjit}\address{Mahidol University International College, Nakhonpathom, Thailand}\email{aram.tan@mahidol.edu}

%\address{Saint Peter's University\\
%                Jersey City, New Jersey\\
%                07306, USA}
%\email{bhopkins@saintpeters.edu}
%\author{Aram Tangboonduangjit}
%\address{Mahidol University\\Mahidol University International College\\
%               Nakhon Pathom\\
%               73170, Thailand}
%\email{aram.tan@mahidol.edu}

\begin{abstract}
In 2013, Joerg Arndt recorded that the Fibonacci numbers count integer compositions where the first part is greater than the second, the third part is greater than the fourth, etc.  We provide a new combinatorial proof that verifies his observation using compositions with only odd parts as studied by De Morgan.  We generalize the descent condition to establish families of recurrence relations related to two types of compositions: those made of any odd part and certain even parts, and those made of any even part and certain odd parts.  These generalizations connect to compositions studied by Andrews and Viennot.  New tools used in the combinatorial proofs include two permutations of compositions and a statistic based on the signed pairwise difference between parts.
\end{abstract}

\keywords{integer compositions, Fibonacci numbers, linear recurrences, combinatorial proofs}

\subjclass{05A17, 11B37, 05A19}

\maketitle

\section{Introduction}
Joerg Arndt observed the following occurrence of the Fibonacci numbers counting a subset of integer compositions.  This appears as a comment in the On-Line Encyclopedia of Integer Sequences \cite[A000045]{o}.

An integer composition of a positive integer $n$ is an ordered collection of parts $(c_1, \ldots, c_t)$ such that $\sum c_i = n$.  When listing compositions with single-digit parts, we often use the condensed representation $c_1 \cdots c_t$, sometimes with exponents denoting repetition.  Let $C(n)$ be the set of all compositions of $n$.  We define Arndt's compositions in terms of pairwise descending parts.

\begin{definition}
Let $A(n) \subset C(n)$ be the compositions such that $c_{2i-1} > c_{2i}$ for each positive integer $i$.  If the number of parts is odd, then the final inequality is vacuously true.
\end{definition}

See Table \ref{t1} for examples.  Arndt recorded that $a(n) = |A(n)| = f_n$, the $n$th Fibonacci number defined by $f_0 = 0$, $f_1=1$, and $f_n = f_{n-1}+f_{n-2}$ for $n\ge 2$, but did not provide a proof (and, per personal communication in April 2022, does not recall how he made the connection).

\begin{table}[h]
\caption{Arndt's compositions and their counts for small values of $n$.} 
\centering
\begin{tabular}{c|l|c}
$n$ & $A(n)$ & $a(n)$ \\ \hline
1 & 1 & 1 \\
2 & 2 & 1 \\
3 & 3, 21 & 2 \\
4 & 4, 31, 211 & 3 \\
5 & 5, 41, 32, 311, 212 & 5 \\
6 & 6, 51, 42, 411, 321, 312, 213, 2121 & 8 \\
7 & 7, 61, 52, 511, 43, 421, 412, 322, 313, 3121, 214, 2131, 21211 & 13
\end{tabular}
\label{t1}
\end{table}

In recent work \cite{ht22}, the authors confirmed Arndt's observation and explored generalizing the pairwise difference condition to $c_{2i-1} > c_{2i} + k$ for positive integers $k$.  Here, we give another proof that $a(n) =  f_n$ which leads to results about generalizing the pairwise difference condition to $c_{2i-1} > c_{2i} + k$ for negative integers $k$.

\begin{definition}
Let $C_\text{odd}(n) \subset C(n)$ be the compositions with parts restricted to odd integers.
\end{definition}

In 1846, Augustus De Morgan established that $c_\text{odd}(n)= |C_\text{odd}(n)| = f_n$.  Since this work seems to be only recently rediscovered, we quote his explanation \cite[pp. 203--204]{d46}.

\begin{quotation}
Required the number of ways in which a number can be compounded of odd numbers, different orders counting as different ways. If $a$ be the number of ways in which $n$ can be so made, and $b$ the number of ways in which $n+1$ can be made, then $a+b$ must be the number of ways in which $n+2$ can be made; for every way of making 12 out of odd numbers is either a way of making 10 with the last number increased by 2, or a way of making 11 with a 1 annexed. Thus, $1+5+3+3$ gives 12, formed from $1+5+3+1$ giving 10. But $1+9+1+1$ is formed from $1+9+1$ giving 11. Consequently, the number of ways of forming 12 is the sum of the number of ways of forming 10 and of forming 11. Now, 1 can only be formed in 1 way, and 2 can only be formed in 1 way; hence 3 can only be formed in $1+1$ or 2 ways, 4 in only $1+2$ or 3 ways. If we take the series 1, 1, 2, 3, 5, 8, 13, 21, 34, 55, 89, \&c.\ in which each number is the sum of the two preceding, then the $n$th number of this set is the number of ways (orders counting) in which $n$ can be formed of odd numbers. Thus, 10 can be formed in 55 ways, 11 in 89 ways, \&c.
\end{quotation}

\section{Another proof of Arndt's observation}

We establish a bijection between Arndt's compositions $A(n)$ and the compositions with odd parts considered by De Morgan.

Our arguments are simplified if we assume that every composition has even length; we achieve this by adding a final part 0 if the normal length of the composition is odd.  The examples, however, will not include any terminal zeros.

\begin{thm} \label{adm}
$A(n) \cong C_\text{odd}(n)$.
\end{thm}

\begin{proof}
Given $c = (c_1, \ldots, c_t) \in A(n)$, convert each pair of parts $(c_{2i-1},c_{2i})$ into $(1^{c_{2i-1}-c_{2i}-1}, 2c_{2i}+1)$, where superscripts denote repetition, a composition in $C_\text{odd}(c_{2i-1}+c_{2i})$.  Concatenating these gives an image in $C_\text{odd}(n)$.  Visually, this uses the ``bar graph'' representation of a composition with each part $c_i$ represented by a column of $c_i$ boxes.  The operation reads two bars from top to bottom where one less than the amount that $c_{2i-1}$ exceeds $c_{2i}$ become 1s and, if $c_{2i} > 0$, the rest contribute an odd number at least 3.  The possibility $c_{2i} = 0$ can only occur in the final pair, in which case the image composition ends in $c_{2i-1}$ parts 1.  Since $c_{2i-1} > c_{2i}$, the specified number of 1s is nonnegative.  See Figure \ref{admfig} for an example.

\begin{figure}[h]
\begin{center}
\setlength{\unitlength}{.5cm}
\begin{picture}(13,7)
\thicklines
\put(0,0){\line(0,1){6}}
\put(1,0){\line(0,1){6}}
\put(2,0){\line(0,1){4}}
\put(3,0){\line(0,1){4}}
\put(4,0){\line(0,1){3}}
\put(5,0){\line(0,1){3}}

\put(0,6){\line(1,0){1}}
\put(0,5){\line(1,0){1}}
\put(0,4){\line(1,0){1}}
\put(0,3){\line(1,0){1}}
\put(0,2){\line(1,0){5}}
\put(0,1){\line(1,0){5}}
\put(0,0){\line(1,0){5}}
\put(2,4){\line(1,0){1}}
\put(2,3){\line(1,0){3}}

\put(9,2){\line(0,1){4}}
\put(11,3){\line(0,1){1}}
\put(13,0){\line(0,1){3}}
{\linethickness{1mm}
\put(8,0){\line(0,1){6}}
\put(10,0){\line(0,1){4}}
\put(12,0){\line(0,1){3}}}

\put(8,6){\line(1,0){1}}
\put(8,5){\line(1,0){1}}
\put(8,4){\line(1,0){1}}
\put(8,3){\line(1,0){1}}
\put(8,0){\line(1,0){5}}
\put(9,2){\line(1,0){1}}
\put(10,4){\line(1,0){1}}
\put(11,3){\line(1,0){2}}
\put(12,2){\line(1,0){1}}
\put(12,1){\line(1,0){1}}

\put(8.3,5.3){1}
\put(8.3,4.3){1}
\put(8.3,3.3){1}
\put(8.8,1.3){5}
\put(10.8,2.3){7}
\put(12.3,2.3){1}
\put(12.3,1.3){1}
\put(12.3,0.3){1}

\end{picture}
\end{center}
\caption{$(6,2,4,3,3) \in A(18)$ corresponds to $(1,1,1,5,7,1,1,1) \in C_\text{odd}(18)$.} \label{admfig}
\end{figure}
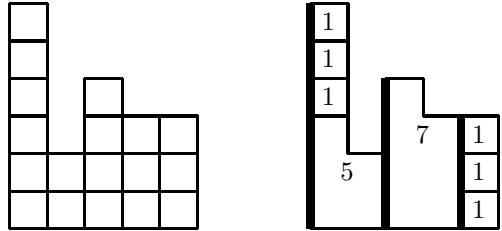

For the reverse map, a composition in $C_\text{odd}(n)$ can be broken into runs of the form $(1^a, 2b+1)$ for nonnegative integers $a$ and $b$ (where $b=0$ can only occur at the end in the case that the last part is 1).  The subsequence $(1^a, 2b+1)$ corresponds to the 2-part composition $(a+b+1,b) \in A(a+2b+1)$ (clearly $a+b+1>b$).  Concatenating the pairs gives a composition in $A(n)$.

It is clear that the two maps are inverses, establishing the bijection.
\end{proof}

See Table \ref{admex} for an example of the correspondence.

\begin{table}[h]
\caption{The correspondence between $A(6)$ and $C_\text{odd}(6)$ from Theorem \ref{adm}.} 
\centering
\begin{tabular}{rcl}
$A(6)$ & $\mapsto$ & $C_\text{odd}(6)$ \\ \hline
6 & & $1^6$ \\
51 & & $1^33$ \\
42 & & 15 \\
411 & & 1131 \\
321 & & 51 \\
312 & & 1311 \\
213 & & $31^3$ \\
2121 & & 33
\end{tabular}
\label{admex}
\end{table}

By De Morgan's $c_\text{odd}(n) = f_n$ result, we have shown again that $a(n) = f_n$.

In relation to the bijection between $A(n)$ and compositions with parts restricted to parts 1 and 2 given in \cite[Thm. 2.3]{ht22}, the bijection here sends all parts 2 and a part 1 to an odd part at least 3.  

The connection to De Morgan's restricted compositions counted by the Fibonacci numbers is not just a nice complement to the result about compositions restricted to parts 1 and 2, it is the base case of the family of bijections for a generalization of Arndt compositions defined in Section \ref{s4}.

\section{Two permutations of compositions}

In this section, we establish two permutations of $C(n)$, i.e., bijections from the set of compositions of $n$ to itself.  These will be used in the next section and may be of independent interest.  Also, we characterize and enumerate the compositions fixed by each permutation.

Both of the permutations connect pairs of parts and sequences of the form $(1^j,\ell)$ with $j \ge 0$ and $\ell \ge 2$, that is, a run of ones followed by a larger part (with the possible exception of a terminal run of ones).

\begin{thm} \label{knegeven}
Given a composition $c \in C(n)$, determine $U(c)$ by the pairwise map
\[ (c_{2i-1},c_{2i}) \mapsto \begin{cases} (1^{c_{2i-1}-c_{2i}-1},2c_{2i}+1) & \text{if $c_{2i}-c_{2i-1} <0$,} \\ (1^{c_{2i-1}+c_{2i}-2k},2k) & \text{if $2k-2 \le c_{2i}-c_{2i-1}  < 2k$ for $k \ge 1$.}\end{cases} \]
The resulting map $U$ is a permutation of $C(n)$.
\end{thm}

See Table \ref{1perm} for examples of $U$.

\begin{table}[h]
\caption{The permutation $U$ of Theorem \ref{knegeven} on $C(5)$ with compositions organized by the $h$ values of Definition \ref{hdef}.}  
\centering
\begin{tabular}{c|c|c}
$h(c) < 0$ & $0 \le h(c) < 2$ & $2 \le h(c) < 4$ \\ \hline
$5 \mapsto 1^5$ & $23 \mapsto 1^32$ & $14 \mapsto 14$ \\
$41 \mapsto 113$ & $221 \mapsto 1121$ & $131 \mapsto 41$ \\
$32 \mapsto 5$ & $2111 \mapsto 32$ \\
$311 \mapsto 131$ & $122 \mapsto 1211$ \\
$212 \mapsto 311$ & $1211 \mapsto 122$ \\
& $113 \mapsto 21^3$ \\
& $1121 \mapsto 23$ \\
& $1^32 \mapsto 212$ \\
& $1^5 \mapsto 221$
\end{tabular}
\label{1perm}
\end{table}

\begin{proof}
Note that a final pair $(c_{2i-1},0)$ maps to $c_{2i-1}$ ones.

Since 
\[c_{2i-1}+c_{2i} = (c_{2i-1}-c_{2i}-1) + 2c_{2i}+1 = (c_{2i-1}+c_{2i}-2k)+2k,\]
clearly $U(c) \in C(n)$.  To show that the map is injective, suppose $c, d \in C(n)$.  If $c \ne d$, then the two compositions must differ in at least one pair of parts and the images of those pairs are distinct.

For the reverse map, partition a composition into subsequences of the form $(1^j,\ell)$ with $j \ge 0$ and $\ell \ge 2$, i.e., parts greater than one each grouped with any preceding run of ones.  There can also be a terminal run of ones, that is, $\ell = 1$ is allowed at the end and a final run $1^j$ is treated as $(1^{j-1},1)$.  The reverse map is determined by
\[ (1^j,\ell) \mapsto \begin{cases} 
 (j+m,m-1) & \text{if $\ell = 2m-1$}, \\
(k+1, k + 2m - 1) & \text{if $j = 2k$ and $\ell = 2m$}, \\
(k, k + 2m - 1) & \text{if $j=2k-1$ and $\ell=2m$}. \end{cases} \]
It is direct to check that $U$ sends $(j+m, m-1)$, since $j+m>m-1$, to $(1^{j+m-m+1-1}, 2(m-1) + 1) = (1^j, 2m-1)$, etc. \end{proof}

Our applications of the permutations depend on the following statistic.

\begin{definition} \label{hdef}
Given a composition $c = (c_1, \ldots, c_{2u})$, let 
\[ h(c) = \max_{1 \le i \le u} (c_{2i} - c_{2i-1}).\]
This is the greatest signed pairwise difference of $c$, the greatest increase from an odd index part to its successor.
\end{definition}

Table \ref{1perm} illustrating the permutation $U$ also gives several examples of the $h$ statistic.  Notice that if $h(c) < 0$, i.e., if all pairs satisfy $c_{2i-1} > c_{2i}$, then $U(c)$ is a composition with only odd parts.  If $h(c) = 2j-2$ or $2j-1$ for $j \ge 1$, then $U(c)$ is a composition with parts odd or from the set $\{2, 4, \ldots, 2j\}$.  

As a permutation, $U$ partitions the compositions $C(n)$ into disjoint cycles.  In a first step towards understanding the resulting cycle structure, we characterize and enumerate the 1-cycles/fixed points $C^U\!(n)$ of $C(n)$ under $U$.  

\begin{prop} \label{UFP}
For the map $U$ described in Theorem \ref{knegeven}, the fixed points $C^U\!(n)$ are the compositions $c = (c_1, \ldots, c_t) \in C(n)$ with the parts $c_{2i-1} = 1$ and parts $c_{2i}$ even for every $i$.  
The count $c^U\!(n)$ satisfies the recurrence $c^U\!(n) = c^U\!(n-2) + c^U\!(n-3)$ with initial values $c^U\!(1) = 1$, $c^U\!(2) = 0$, $c^U\!(3)=1$. 
\end{prop}

See Table \ref{Ufixed} for the $C^U\!(n)$ compositions through $n = 10$.  The sequence $c^U\!(n)$ is known as the Padovan numbers \cite[A000931]{o}.

\begin{table}[h]
\centering
\caption{The compositions in $C(n)$ fixed by the permutation $U$ and their counts for small $n$.} 
\begin{tabular}{r|l|l}
$n$ & $C^U\!(n)$ & $c^U\!(n)$ \\ \hline
1 & 1 & 1\\
2 & $\varnothing$ & 0\\
3 & 12 & 1\\
4 & 121 & 1\\
5 & 14 & 1\\
6 & 141, 1212 & 2\\
7 & 16, 12121 & 2\\
8 & 161, 1412, 1214 & 3 \\
9 & 18, 14121, 12141, 121212 & 4\\
10 & 181, 1612, 1414, 1216, 1212121 & 5
\end{tabular}
\label{Ufixed}
\end{table}

\begin{proof}
Examining the definition of $U$, a pair of parts $(c_{2i-1},c_{2i})$ fixed by the permutation must have $c_{2i-1} = 1$.  In the first case of the definition of $U$, to have a single part 1 start the image requires $c_{2i-1} = c_{2i} + 2$, a contradiction.  In the second case of the definition of $U$, having a single part 1 in the image requires $c_{2i-1}+c_{2i} = 2k + 1$.  Thus a composition in $C^U\!(n)$ consists of pairs $(1, 2k)$ where a final odd-indexed part 1 is allowed (the case $c_{2i}=0$).

For the enumeration result, we establish a bijection $C^U\!(n) \cong C^U\!(n-2) \cup C^U\!(n-3)$.

Given a $U$-fixed composition of $n-3$, add two parts at the end of the composition: $1,2$ if the original length is even, $2,1$ if the original length is odd.  The resulting composition of $n$ has an additional even-indexed part 2 and an additional odd-indexed part 1, so it is in $C^U\!(n)$.  Given a $U$-fixed composition of $n-2$, increase the last nonzero even-indexed part by 2.  This gives a composition of $n$ that maintains the necessary structure, so it is also in $C^U\!(n)$.  The two sets are disjoint considering the last nonzero even part: it is 2 for the compositions coming from $C^U\!(n-3)$, while it is at least 4 for the compositions coming from $C^U\!(n-2)$.

For the reverse map, given a $U$-fixed composition of $n$, if the last nonzero even part is at least 4, decrease it by 2 to make a composition in $C^U\!(n-2)$; if the last nonzero even part is 2, remove it and the last 1 (whether before or after the 2) to make a composition in $C^U\!(n-3)$.  

It is clear that the two maps are inverses, establishing the bijection.  The initial values follow from Table \ref{Ufixed}.
\end{proof}

See Table \ref{Uex} for an example of the bijection.

\begin{table}[h]
\caption{The correspondence between $C^U\!(8) \cup C^U\!(7)$ and $C^U\!(10)$ from Proposition \ref{UFP}.} 
\centering
\begin{tabular}{rcl}
$C^U\!(8) \cup C^U\!(7)$ & $\mapsto$ & $C^U\!(10)$ \\ \hline
161 & & 181 \\
1412 & & 1414 \\
1214 & & 1216 \\ \cline{1-1}
16 & & 1612 \\
12121 & & 1212121 
\end{tabular}
\label{Uex}
\end{table}

The second permutation is very similar to $U$, switching the parity of the second part of each pair.

\begin{thm} \label{knegodd}
Given a composition $c \in C(n)$, determine $V(c)$ by the pairwise map
\[ (c_{2i-1},c_{2i}) \mapsto \begin{cases} (1^{c_{2i-1}-c_{2i}},2c_{2i}) & \text{if $c_{2i}-c_{2i-1} < 1$,} \\ (1^{c_{2i-1}+c_{2i}-2k-1},2k+1) & \text{if $2k-1 \le c_{2i}-c_{2i-1} < 2k+1$ for $k \ge 1$.}\end{cases} \]
The resulting map $V$ is a permutation of $C(n)$.
\end{thm}

See Table \ref{2perm} for examples of $V$.

\begin{table}[h]
\centering
\caption{The permutation $V$ on $C(5)$ with compositions organized by $h$ values.} 
\begin{tabular}{c|c|c}
$h(c) < 1$ & $1 \le h(c) < 3$ & $3 \le h(c) < 5$ \\ \hline
$5 \mapsto 1^5$ & $23 \mapsto 113$ & $14 \mapsto 5$ \\
$41 \mapsto 1^32$ & $131 \mapsto 131$ \\
$32 \mapsto 14$ & $122 \mapsto 311$ \\
$311 \mapsto 1121$ & $1211 \mapsto 32$ \\
$221 \mapsto 41$ & $1112 \mapsto 23$\\
$212 \mapsto 1211$ & \\
$21^3 \mapsto 122$ &\\
$113 \mapsto 21^3$ &\\
$1121 \mapsto 212$ &\\
$1^5 \mapsto 221$ &
\end{tabular}
\label{2perm}
\end{table}

\begin{proof}
The proof is analogous to the proof of Theorem \ref{knegeven}.  Here, we just give the explicit reverse map, again in terms of subsequences of the form $(1^j,\ell)$ with $j \ge 0$ and $\ell \ge 2$ where a terminal $1^j$ is treated as $(1^{j-1},1)$.  Applying
\[ (1^j,\ell) \mapsto \begin{cases} (j+m,m) & \text{if $\ell = 2m$}, \\
(k+1, k + 2m - 2) & \text{if $j = 2k$ and $\ell = 2m-1$}, \\
(k, k + 2m - 2) & \text{if $j=2k-1$ and $\ell=2m-1$}. \end{cases}\]
to each subsequence of $c$ determines the reverse map.
\end{proof}

Notice that if $h(c) = 2j-1$ or $2j$ for $j \ge 1$, then the image is a composition with parts even or from the set $\{1, 3, \ldots, 2j+1\}$.  

We also characterize and enumerate $C^V\!(n)$, the $V$-fixed compositions of $n$.

\begin{prop} \label{VFP}
For the map $V$ described in Theorem \ref{knegodd}, the fixed points $C^V\!(n)$ are the compositions $c = (c_1, \ldots, c_t) \in C(n)$ with the parts $c_{2i-1} = 1$ and parts $c_{2i}$ odd and at least 3 for every $i$.
The count $c^V\!(n)$ satisfies the recurrence $c^V\!(n) = c^V\!(n-2)+ c^V\!(n-4)$ for $n \ge 5$ with initial values $c^V\!(1) = c^V\!(4) = 1$ and $c^V\!(2) = c^V\!(3) = 0$. 
\end{prop}

See Table \ref{Vfixed} for the $C^V\!(n)$ compositions through $n = 10$.  The sequence $c^V\!(n)$ from $n = 1$ is the Fibonacci numbers each occurring twice.  In other words, $c^V\!(2j) = c^V\!(2j+1) = f_{j-1}$ for $j \ge 1$.

\begin{table}[h]
\centering
\caption{The compositions in $C(n)$ fixed by the permutation $V$ and their counts for small $n$.} 
\begin{tabular}{r|l|l}
$n$ & $C^V\!(n)$ & $c^V\!(n)$ \\ \hline
1 & 1 & 1 \\
2 & $\varnothing$ & 0 \\
3 & $\varnothing$ & 0 \\
4 & 13 & 1\\
5 & 131 & 1\\
6 & 15  & 1\\
7 & 151 & 1\\
8 &  17, 1313 & 2\\
9 &  171, 13131 & 2\\
10 & 19, 1513, 1315 & 3 
\end{tabular}
\label{Vfixed}
\end{table}

\begin{proof}
Examining the definition of $V$, a pair of parts $(c_{2i-1},c_{2i})$ fixed by the permutation must have $c_{2i-1} = 1$.  In the first case of the definition of $V$, to have a single part 1 start the image requires $c_{2i-1} = c_{2i} + 1$ which is only true if $c_{2i} = 0$.  In the second case of the definition of $V$, having a single part 1 in the image requires $c_{2i-1}+c_{2i} = 2k + 2$ (recall $k \ge 1$).  Thus a composition in $C^V\!(n)$ consists of pairs $(1, 2k+1)$ where a final odd-indexed part 1 is allowed (from the first case of the definition when $c_{2i-1} = 1$ and $c_{2i}=0$).

For the enumeration result, we establish bijections $C^V\!(2j+1) \cong C^V\!(2j)$ and $C^V\!(2j) \cong C^V\!(2j-2) \cup C^V\!(2j-4)$.  Note that since all parts of a composition in $C^V\!(n)$ are odd, the compositions in $C^V\!(2j+1)$ have odd length while the compositions in $C^V\!(2j)$ have even length.

For the first bijection, a composition in $C^V\!(2j+1)$ must have last part 1; removing that gives a composition in $C^V\!(2j)$.  For the reverse map, given a composition in $C^V\!(2j)$, add a part 1 at the end.  It is clear that the maps are inverses, establishing the first bijection.

For the second bijection, given a $V$-fixed composition of $2j - 4$, add the parts 1, 3 at the end of the composition.  The resulting composition of $2j$ has an additional odd-indexed part 1 and an additional even-indexed part 3, so it is in $C^V\!(2j)$.  Given a $V$-fixed composition of $2j-2$, increase the last part by 2.  This gives a composition of $2j$ with the necessary structure and final part odd at least 5, so it is in $C^V\!(2j)$.  The two sets are disjoint considering the last part: it is 3 for compositions coming from $C^V\!(2j-4)$ while it is odd and at least 5 for compositions coming from $C^V\!(2j-2)$.

For the reverse map of the second bijection, given a $V$-fixed composition of $2j$, if the last part is at least 5, decrease it by 2 to make a composition in $C^V\!(2j-2)$; if the last part is 3, remove it and the last part 1 to make a composition in $C^V\!(2j-4)$.  It is clear that the maps are inverses, establishing the second bijection.

The initial values follow from Table \ref{Vfixed}.
\end{proof}

See Table \ref{Vex} for an example of the bijection.

\begin{table}[h]
\centering
\caption{The correspondence between two copies of $C^V\!(6)$ and $C^V\!(8)$ from Proposition \ref{VFP}.} 
\begin{tabular}{rcl}
$2C^V\!(6)$ & $\mapsto$ & $C^V\!(8)$ \\ \hline
15 & & 17 \\
1311 & & 1313 \\
$1^33$ & & $1^35$ \\
$1^6$ & & $1^53$ \\ \cline{1-1}
15 & & 1511 \\
1311 & & $131^4$ \\
$1^33$ & & $1^3311$ \\
$1^6$ & & $1^8$
\end{tabular}
\label{Vex}
\end{table}

\section{Generalizing Arndt's compositions} \label{s4}

Arndt's compositions $A(n)$ require a descent from each $c_{2i-1}$ to $c_{2i}$ for each $i$.  Repeating \cite[Def. 3.1]{ht22}, this condition can be generalized as follows.

\begin{definition} \label{ga}
Given an integer $k$, let $A(n,k) \subset C(n)$ be the compositions such that $c_{2i-1} > c_{2i}+k$ for each positive integer $i$.  If the number of parts is odd, then the final inequality is vacuously true.
\end{definition}

The Arndt compositions $A(n)$ are the same as $A(n,0)$.  See Table \ref{counts} for values of $a(n,k) = |A(n,k)|$ for small values.  The authors considered the cases $k > 0$ in \cite{ht22} where a steeper descent is required.  Here we focus on the cases $k < 0$ which allows for equality or a limited increase between each part $c_{2i-1}$ and $c_{2i}$.

\begin{table}[h]
\centering
\caption{Generalized Arndt compositions counts $a(n,k)$ for small $n$ and $k$.} 
\begin{tabular}{r|rrrrrrrrrr}
$k \backslash n$ & 1 & 2 & 3 & 4 & 5 & 6 & 7 & 8 & 9 & 10 \\ \hline
3 & 1 & 1 & 1 & 1 & 1 & 2 & 3 & 5 & 7 & 10 \\
2 & 1 & 1 & 1 & 1 & 2 & 3 & 5 & 7 & 10 & 14 \\
1 & 1 & 1 & 1 & 2 & 3 & 5 & 7 & 11 & 16 & 25 \\ \hline
0 & 1 & 1 & 2 & 3 & 5 & 8 & 13 & 21 & 34 & 55 \\ \hline
$-1$ & 1 & 2 & 3 & 6 & 10 & 19 & 33 & 61 & 108 & 197 \\
$-2$ & 1 & 2 & 4 & 7 & 14 & 26 & 50 & 95 & 181 & 345 \\
$-3$ & 1 & 2 & 4 & 8 & 15 & 30 & 58 & 114 & 222 & 435
\end{tabular}
\label{counts}
\end{table} 

As an additional motivation, the generalized Arndt compositions $A(n,k)$ for negative $k$ are related to the compositions studied by George Andrews in \cite{a81} that require $c_i \ge c_{i+1} - d$ for all $i$ and an integer parameter $d$.  These $d$-compositions were further studied by Viennot \cite{v87} using his ``heaps of pieces'' approach.
The generalized Arndt compositions are a pairwise version of this with $c_{2i-1} > c_{2i}+k$ (recall $k$ is negative) but no restriction between $c_{2i}$ and $c_{2i+1}$.

\begin{thm} \label{genArecur}
The number of generalized Arndt compositions satisfies the recurrence
\[ a(n,k) = \begin{cases}	a(n-1, k) + 2a(n-2, k) - a(n-2+k, k) & \text{if $k \le 0$}, \\
a(n-1,k) + a(n-2,k) - a(n-3,k) + a(n-3-k,k) & \text{if $k \ge 0$}. \end{cases} \]
\end{thm}

\begin{proof}
The $k = 0$ case is Theorem \ref{adm} 
%3 
above (see also Theorems 2.1 and 2.3 of \cite{ht22}).  The $k > 0$ case was established in \cite[Thm.\ 3.2]{ht22}.  Here, suppose $k < 0$.

Write $A^o(n,k)$ for the compositions of $A(n,k)$ with odd length and $A^e(n,k)$ for those with even length.  We establish two bijections,
\begin{gather*} A^o(n,k) \cong A(n-1,k), \\ A^e(n,k) \cup A(n-2+k,k) \cong 2A(n-2,k) \end{gather*}
where the coefficient 2 indicates two copies of the set $A(n-2,k)$.  This gives a slightly stronger result,
\begin{gather*} a^o(n,k) = a(n-1,k), \\ a^e(n,k) = 2a(n-2,k) - a(n-2+k,k) \end{gather*}
from which the claim follows, since $a(n,k) = a^o(n,k) + a^e(n,k)$.

For the first bijection, given $c = (c_1, \ldots, c_t) \in A^o(n,k)$, let
\[c \mapsto (c_1, \ldots, c_{t-1}, c_t - 1)\]
where, if $c_t = 1$, the resulting 0 is omitted in the image.  Since decreasing (or deleting) the terminal odd indexed part does not affect the pairwise difference conditions, this gives a composition in $A(n-1,k)$.  For the reverse map, given $c = (c_1, \ldots, c_t) \in A(n-1,k)$, let 
\[c \mapsto \begin{cases} (c_1, \ldots, c_t+1) & \text{if $t$ odd}, \\ (c_1, \ldots, c_t,1) & \text{if $t$ is even}, \end{cases} \]
in other words, adding 1 in the last possible odd position.  Note that in either case, the image has odd length.  Again, the terminal odd indexed part does not affect the pairwise difference conditions, so the image is in $A^o(n,k)$.  These maps are clearly inverses.

See Table \ref{bij1} for an example of this first bijection.

\begin{table}[h]
\centering
\caption{An example of the first bijection in the proof of Theorem \ref{genArecur} between $A^o(6,-1)$ and $A(5,-1)$.} 
\begin{tabular}{rcl}
$A^o(6,-1)$ & $\mapsto$ & $A(5,-1)$ \\ \hline
6 && 5 \\
411 && 41 \\
321 && 32 \\
312 && 311 \\
$2^3$ && 221 \\
213 && 212 \\
$21^4$ && $21^3$ \\
114 && 113 \\
11211 && 1121 \\
$1^42$ && $1^5$
\end{tabular}
\label{bij1}
\end{table}

For the second bijection, suppose $c = (c_1, \ldots, c_t) \in A^e(n,k)$.  Let 
\[ c \mapsto (c_1, \ldots,c_{t-1}-1, c_t-1)\]
where any parts 0 (from $c_t = 1$ or $c_{t-1}=1$) are omitted.  This gives compositions in $A(n-2,k)$ since decreasing the last two parts by 1 maintains the pairwise difference condition or, when $c_{t-1}=1$, gives an odd length composition with last part $c_t-1$.  More specifically, the compositions with $c_t = 1$ give one set of $A(n-2,k)$ compositions (note that the images can have even or odd length).  The compositions with $c_t > 1$ give some compositions in the second $A(n-2,k)$.  Again, the images can have odd or even length depending on $c_{t-1}$.  In the case of an odd length image, we have $c_{t-1} = 1$ and, by the pairwise difference condition, $c_t < -k + 1$, so that the last part in the image is less than $-k$.  

Continuing the second bijection, suppose $c = (c_1, \ldots, c_t) \in A(n-2+k,k)$.  Let
\[c \mapsto \begin{cases} (c_1, \ldots, c_t-k) & \text{if $t$ odd}, \\ (c_1, \ldots, c_t,-k) & \text{if $t$ is even}, \end{cases} \]
in other words, adding $-k$ in the last possible odd position.  This does not affect the pairwise difference conditions, so the images are in $A(n-2,k)$, have odd length, and last part at least $-k$.

The only case to verify that this is an injection are the odd length compositions of $A(n-2,k)$.  As noted, those with last part less than $-k$ come from $A^e(n,k)$ while those with last part at least $-k$ come from $A(n-2+k,k)$.

We conclude the proof by giving the reverse map for the second bijection.  Suppose $c = (c_1, \ldots, c_t)$ in the first copy of $A(n-2,k)$.  Let
\[ c \mapsto \begin{cases} (c_1,\ldots, c_t+1,1) & \text{if $t$ odd}, \\ (c_1,\ldots, c_{t-1}+1, c_t +1) & \text{if $t$ even}, \end{cases} \]
in other words, make an even length composition of $n$ by adding 1 in the final odd indexed and even indexed parts of the pre-image.  The image is in $A^e(n,k)$: for $t$ odd, the image has penultimate part greater than 1 and last part 1, which certainly satisfies the pairwise difference condition; for $t$ even, adding 1 to the last two parts maintains the pairwise difference condition.  This gives the $A^e(n,k)$ compositions with penultimate parts greater than 1.

Continuing the reverse map for the second bijection, suppose $c = (c_1, \ldots, c_t)$ in the second copy of $A(n-2,k)$.  Let
\[ c \mapsto \begin{cases} (c_1,\ldots, c_{t-1},1,c_t+1) & \text{if $t$ odd and $c_t < -k$}, \\ 
(c_1,\ldots, c_{t-1}, c_t+k) & \text{if $t$ odd and $c_t \ge -k$},\\
(c_1,\ldots, c_t, 1, 1) & \text{if $t$ even}. \end{cases} \]
The $t$ odd and $c_t < -k$ case gives an even length composition of $n$ that satisfies the pairwise difference condition since $1 = -k+1+k > c_t + 1 + k$.  These are the compositions of $A^e(n,k)$ with penultimate part 1 and last part greater than 1.  The $t$ even case gives the compositions in $A^e(n,k)$ whose last two parts are 1 (which certainly satisfies the pairwise difference condition).  With the images from the first copy of $A(n-2,k)$ above, these give all of $A^e(n,k)$.  The $t$ odd and $c_t \ge -k$ case gives compositions of $n-2+k$ of both odd and even length (the latter when $c_t = -k$) which are in $A(n-2+k,k)$ since the first $t-1$ parts satisfy the pairwise difference condition.

As noted, the reverse map is injective by looking at the last two parts of the images: penultimate part greater than 1 from the first copy of $A(n-2,k)$, penultimate part 1 from some of the second copy of $A(n-2,k)$ with the remainder of the second copy going to $A(n-2+k,k)$.
\end{proof}

See Table \ref{bij2} for an example of this second bijection.

\begin{table}[h]
\centering
\caption{An example of the second bijection in the proof of Theorem \ref{genArecur}, from $A^e(6,-1) \cup A(3,-1)$ on the left-hand side, from $2A(4,-1)$ on the right-hand side.} 
\begin{tabular}{rcl||rcl}
$A^e(6,-1) \cup A(3,-1)$ & $\mapsto$ & $2A(4,-1)$ & $2A(4,-1)$ & $\mapsto$ & $A^e(6,-1) \cup A(3,-1)$ \\ \hline
51 && 4 & 4 && 51 \\
42 && 31 & 31 && 42 \\
33 && 22 & 22 && 33 \\
$31^3$ && 31 & 211 && 2121 \\
1131 && 112 & 112 && 1131 \\
2211 && 22 & $1^4$ && 1122 \\ \cline{4-6}
2121 && 211 & 4 && 3 \\
1122 && $1^4$ & 31 && $31^3$ \\
$1^6$ && $1^4$ & 22 && 2211 \\ \cline{1-3}
3 && 4 & 211 && 21 \\
21 && 211 & 112 && $1^3$ \\
$1^3$ && 112 & $1^4$ && $1^6$
\end{tabular}
\label{bij2}
\end{table}

We conclude this work by establishing a bijection between the generalized Arndt compositions $A(n,k)$ for $k < 0$ and classes of compositions based on the parity of their parts defined next.

\begin{definition}
For a nonnegative integer $k$, let $C_\text{odd}^{2k}(n)$ be the compositions of $n$ with parts restricted to positive odd integers and the even integers $2, \ldots, 2k$.  Similarly, let $C_\text{even}^{2k+1}(n)$ be the compositions of $n$ with parts restricted to positive even integers and the odd integers $1, \ldots, 2k+1$.
\end{definition}

Note that $C_\text{odd}^0(n)$ are the odd part compositions considered by De Morgan.  Current OEIS occurrences for other sequences of these counts are $c_\text{even}^{1}(n)$ \cite[A028495]{o} and $c_\text{odd}^{2}(n)$ \cite[A052535]{o}.  For fixed $n$ and sufficiently large $k$, we have $C_\text{odd}^{2k}(n) = C_\text{even}^{2k+1}(n) = C(n)$, as all parts arising in $C(n)$ are allowed.

\begin{thm} \label{oddsmalleven}
For $k$ a negative even number, $a(n,k) = c_\text{odd}^{-k}(n)$.
\end{thm}

\begin{proof}
Given $k = -2j$ for a positive integer $j$, we show that the permutation $U$ of Theorem \ref{knegeven} restricted to $A(n,-2j)$ is a bijection with $C_\text{odd}^{2j}(n)$.

Comparing the Definition \ref{ga} of $A(n,k)$ and Definition \ref{hdef} of the $h$ statistic, a composition $c \in A(n,-2j)$ has $h(c) < 2j$.  By the description of $U$ in Theorem \ref{knegeven}, compositions with $h(c) < 2j$ are sent to compositions with even parts at most $2j$ or odd parts, i.e., compositions in $C_\text{odd}^{2j}(n)$.  The inverse of $U$ described in the proof of Theorem \ref{knegeven} sends a composition in $C_\text{odd}^{2j}(n)$ to a composition in $A(n,-2j)$, completing the bijection.
\end{proof}

Table \ref{1perm} provides examples for $n=5$.  The 5 compositions in the left-hand column show $A(5,0) \cong C_\text{odd}^{0}(5) = C_\text{odd}(5)$.  The union of the 14 compositions in the left-hand and middle columns show $A(5,-2) \cong C_\text{odd}^{2}(5)$.  The 16 total compositions in all three columns show $A(5,-4) \cong C_\text{odd}^{4}(5) = C(5)$.

\begin{thm} \label{evensmallodd}
For $k$ a negative odd number, $a(n,k) = c_\text{even}^{-k}(n)$.
\end{thm}

\begin{proof}
Analogous to the previous theorem, given $k = -2j-1$ for a positive integer $j$, the permutation $V$ of Theorem \ref{knegodd} restricted to $A(n,-2j-1)$ is a bijection with $C_\text{even}^{2j+1}(n)$.
\end{proof}

Table \ref{2perm} provides examples for $n=5$.  The 10 compositions in the left-hand column show $A(5,-1) \cong C_\text{even}^{1}(5)$.  The union of the 15 compositions in the left-hand and middle columns show $A(5,-3) \cong C_\text{even}^{3}(5)$.  The 16 total compositions in all three columns show $A(5,-5) \cong C_\text{even}^{5}(5) = C(5)$.

The connections established in these last two theorems show why $c_\text{odd}^{2k}(n)$ and $c_\text{even}^{2k+1}(n)$ satisfy the same recurrence relation, a fact that is not obvious from their definitions.

\begin{cor}
For $k \ge 1$,
\begin{gather*}
c_\text{even}^{2k+1}(n) = c_\text{even}^{2k+1}(n-1) + 2c_\text{even}^{2k+1}(n-2) - c_\text{even}^{2k+1}(n-2-2k-1),\\
c_\text{odd}^{2k}(n) = c_\text{odd}^{2k}(n-1) + 2c_\text{odd}^{2k}(n-2) - c_\text{odd}^{2k}(n-2-2k).
\end{gather*}
\end{cor}

\begin{proof}
Theorems \ref{oddsmalleven} and \ref{evensmallodd} connect $c_\text{odd}^{2k}(n)$ and $c_\text{even}^{2k+1}(n)$ to $a(n,-2k-1)$ and $a(n,-2k)$, respectively.  The recurrences follow from Theorem \ref{genArecur}.
\end{proof}

\end{document}